\documentclass[10pt]{amsart}

\usepackage{amssymb,amsmath,amsthm}
\usepackage[arc,all]{xy}
\usepackage{eucal}
\usepackage{enumitem}
\usepackage{url}
\usepackage{scrextend}

\theoremstyle{plain}

\newtheorem{lem}[subsection]{Lemma}
\newtheorem{prop}[subsection]{Proposition}
\newtheorem{thm}[subsection]{Theorem}

\newtheorem{assumption}[subsection]{Assumption}

\theoremstyle{definition}
\newtheorem{defn}[subsection]{Definition}

\theoremstyle{remark}
\newtheorem{rem}[subsection]{Remark}

\newcommand{\Mod}{{ \mathsf{Mod} }}
\newcommand{\ModR}{{ \mathsf{Mod}_\capR }}

\newcommand{\Alg}{{ \mathsf{Alg} }}

\newcommand{\TQ}{{ \mathsf{TQ} }}

\newcommand{\AlgM}{{\Alg_\mathcal{M} }}
\newcommand{\AlgN}{{\Alg_\mathcal{N} }}

\newcommand{\AlgJ}{{ \Alg_J }}

\newcommand{\AlgO}{{ \Alg_\capO }}

\newcommand{\capO}{{ \mathcal{O} }}

\newcommand{\capR}{{ \mathcal{R} }}

\newcommand{\capM}{{\mathcal{M} }}
\newcommand{\capN}{{\mathcal{N} }}

\newcommand{\id}{{ \mathrm{id} }}

\newcommand{\iso}{{ \cong }}

\DeclareMathOperator*{\hocolim}{hocolim}

\DeclareMathOperator*{\holim}{holim}

\DeclareMathOperator{\BAR}{Bar}

\DeclareMathOperator{\im}{im}
\DeclareMathOperator{\diag}{diag}

\newcommand{\abs}[1]{\left\vert#1\right\vert}
\newcommand{\sett}[1]{\left\{#1\right\}}

\title[$\TQ$-completion and the Taylor tower of the identity]{$\TQ$-completion and the Taylor tower of the identity functor}

\author{Nikolas Schonsheck}

\address{Department of Mathematical Sciences, University of Delaware}

\email{nischon@udel.edu}

\begin{document}
\maketitle
\begin{abstract}
The goal of this short paper is to study the convergence of the Taylor tower of the identity functor in the context of operadic algebras in spectra. Specifically, we show that if $A$ is a $(-1)$-connected $\capO$-algebra with $0$-connected $\TQ$-homology spectrum $\TQ(A)$, then there is a natural weak equivalence $P_\infty(\id)A \simeq A^\wedge_\TQ$ between the limit of the Taylor tower of the identity functor evaluated on $A$ and the $\TQ$-completion of $A$. Since, in this context, the identity functor is only known to be $0$-analytic, this result extends knowledge of the Taylor tower of the identity beyond its ``radius of convergence.''
\end{abstract}

\section{Introduction}

Working in the context of symmetric spectra \cite{Hovey_Shipley_Smith,Schwede_book_project} or, more generally, modules over a commutative ring spectrum $\capR$, we consider any algebraic structure in the closed symmetric monoidal category $(\ModR, \wedge, \capR)$ that can be described by a reduced operad $\capO$; that is, $\capO[0]=\ast$ and, hence, $\capO$-algebras are non-unital. See \cite{EKMM} for another construction of a well-behaved category of spectra. The aim of this short paper is to study convergence properties of the Taylor tower of the identity functor in the category $\AlgO$ of $\capO$-algebras.

For any homotopy functor $F$ from spaces to spaces or spectra, Goodwillie constructs \cite{Goodwillie_calculus_1,Goodwillie_calculus_2,Goodwillie_calculus_3} universal $n$-excisive approximations $P_nF$ of $F$, which arrange in a tower as below. 
\begin{align}\label{eqn_generic_taylor_tower}
\xymatrix@R=3ex{
&F \ar[dl] \ar[d] \ar[dr]\\
P_1F & P_2F \ar[l] & P_3F \ar[l] & \cdots\ar[l]
}
\end{align}
For a fixed space $X$, one defines the functor $P_\infty F$ objectwise by $P_\infty F(X) = \holim_nP_nF(X)$. One of the useful properties of this tower is that for suitably nice functors $F$ and sufficiently connected spaces $X$, one has $F(X) \simeq P_\infty F(X)$. In this case, the tower recovers $F$ from its $n$-excisive approximations $P_nF$. Similar results are obtained in the context of $\capO$-algebras in, for instance, \cite{Pereira_general_context} and \cite{Pereira_spectral_operad}.

The Taylor tower of the identity functor in spaces has been widely studied and is known to contain a great deal of non-trivial information and structure; see, for instance \cite{Arone_Mahowald}, \cite{Behrens_EHP_sequence}, \cite{Ching_bar_constructions}, and \cite{Johnson_derivatives}. It is shown in \cite[4.3]{Goodwillie_calculus_2} that the identity functor in spaces is $1$-analytic and, hence, that the Taylor tower of the identity converges strongly on all $1$-connected spaces. In the context of $\capO$-algebras, it follows from \cite[1.6]{Ching_Harper} that the identity functor in $\AlgO$ is $0$-analytic and, hence, its Taylor tower converges strongly on all $0$-connected $\capO$-algebras. Additional structure possessed by this tower is explored in \cite{Clark_derivatives}.

The goal of this paper is to study, in $\AlgO$, the behavior of the Taylor tower of the identity functor when evaluated on $\capO$-algebras which are not assumed to be $0$-connected. In particular, our main theorem applies to any $(-1)$-connected cofibrant $\capO$-algebra with $0$-connected $\TQ$-homology. To keep this paper appropriately concise, we briefly review $\TQ$-homology and completion in Section \ref{sec_background_preliminaries}, but refer the reader to \cite{Ching_Harper_derived_Koszul_duality}, \cite{Harper_modules_over_operads}, \cite{Harper_Hess}, and \cite[Section 3]{Schonsheck_fibration_theorems} for further discussion. In short, the $\TQ$-completion $A_\TQ^\wedge$ of an $\capO$-algebra $A$ is defined as the homotopy limit of the canonical cosimplicial resolution of $A$ obtained by iterating the $\TQ$-Hurewicz map $A \to UQ(A) \simeq \TQ(A)$, and can be thought of as a spectral algebra analogue of Bousfield-Kan's $\mathbb{Z}$-completion \cite[I.4]{Bousfield_Kan} for spaces. The following is our main result.

\begin{thm}\label{thm_main_theorem}
Given a $(-1)$-connected cofibrant $\capO$-algebra $A$, if $\TQ(A)$ is $0$-connected, then there is a natural weak equivalence $P_\infty(\id)A \simeq A_\TQ^\wedge$.
\end{thm}

We also prove, in Section \ref{sec_sufficient_condition}, the following sufficient condition under which $\TQ(A)$ is $0$-connected.

\begin{thm}\label{thm_sufficient_condition}
Given a $(-1)$-connected cofibrant $\capO$-algebra $A$, if for some $k \geq 2$, the map
\begin{align}
\pi_0\big( \capO[k]\wedge_{\Sigma_k} A^{\wedge k}	\big) \to \pi_0(A)
\end{align}
induced by the $\capO$-algebra structure of $A$ is a surjection, then $\TQ(A)$ is $0$-connected.
\end{thm}

\begin{rem}\label{rem_tqa_0_conn_implies_completion_is_too}
If we suppose that $\capO[1] = \capR$, then \cite[1.14]{Harper_Hess} gives an identification
\begin{align}
D_n(\id) \simeq \capO[n]\wedge_{\Sigma_n}\TQ^{\wedge n}
\end{align}
of the layers of the Taylor tower of the identity in $\AlgO$, where the smash products above are understood to be derived (see also \cite[Section 2.7]{Kuhn_Pereira}). Using this identification, it follows from \cite[4.32]{Harper_Hess} that, under the assumptions of Theorem \ref{thm_main_theorem}, the connectivity of $D_n(\id)$ grows without bound as $n$ increases, which in turn implies that the Goodwillie spectral sequence associated to the identity functor converges strongly in the sense of \cite[1.8]{Ching_Harper_derived_Koszul_duality}. Furthermore, by considering the associated $\lim^1$ short exact sequence, one can show that $P_\infty(\id)A\simeq A^\wedge_\TQ$ is $0$-connected. Hence, if $A$ is as in Theorem \ref{thm_main_theorem} and $\pi_0(A) \neq 0$, then $A \not\simeq P_\infty(\id)A$; in other words, it is not the case that $P_\infty(\id)A\simeq A^\wedge_\TQ$ simply because both are weakly equivalent to $A$.
\end{rem}

\noindent
{\bf Relationship to previous work.} It is known \cite[4.3]{Goodwillie_calculus_2} that the identity functor on spaces converges strongly on all 1-connected spaces. Arone-Kankaanrinta \cite[Appendix A]{Arone_Kankaanrinta} are able to determine the behavior of the Taylor tower more generally, showing that for any $0$-connected space $X$, there is a weak equivalence $P_\infty(\id)X \simeq X_{\Omega^\infty\Sigma^\infty}^\wedge$. In the setting of $\capO$-algebras, it follows from \cite[1.6]{Ching_Harper} that the identity functor converges strongly on all $0$-connected $\capO$-algebras. Similarly to Arone-Kankaanrinta, this paper studies the behavior of the Taylor tower of the identity outside of its ``radius of convergence,'' but in the context of $\capO$-algebras. It follows from \cite[2.21]{Kuhn_Pereira} (see also \cite[4.1]{Ogle_Schonsheck}) that $\TQ\simeq P_1(\id) \simeq \hocolim_n\Omega^n\Sigma^n = \Omega^\infty\Sigma^\infty$ as functors on $\AlgO$. Hence, the main result of this paper can be understood as a partial analogue to the above result of Arone-Kankaanrinta.

Our strategy of attack is similar to \cite{Arone_Kankaanrinta} in that we resolve the identity functor by iterates of stabilization, and then analyze the Taylor towers of these iterates. In particular, the bi-tower in \eqref{eqn_the_big_diagram} is the $\capO$-algebra analogue of the main construction used in the proof of \cite[A.1]{Arone_Kankaanrinta}. Our analysis departs from that of Arone-Kankaanrinta in the handling of the columns of this bi-tower.

It is perhaps also worth noting the following. The assumption that $\TQ(A)$ is $0$-connected in Theorem \ref{thm_main_theorem} can be thought of as analogous to assuming, for a connected space $X$, that $\pi_1(\widetilde{\mathbb{Z}}X) \ \iso \ \tilde{H}_1(X;\mathbb{Z}) = 0$ or, equivalently, that $\pi_1(X)$ is a perfect group. Bousfield-Kan show \cite[VII.3.2]{Bousfield_Kan} that, under this assumption, the $\mathbb{Z}$-completion $X_\mathbb{Z}^\wedge$ of $X$ is simply connected, and hence is itself $\mathbb{Z}$-complete by \cite[III.5.4]{Bousfield_Kan}. In this case, a straightforward retract argument as in \cite[I.5.2]{Bousfield_Kan} shows that the natural map $X \to X_\mathbb{Z}^\wedge$ induces an isomorphism on $\mathbb{Z}$-homology, i.e., that $X$ is ``$\mathbb{Z}$-good.'' An analogous result follows from our main theorems, under the additional assumption that $\capO[1] = \capR$. As explained in Remark \ref{rem_tqa_0_conn_implies_completion_is_too}, if $\TQ(A)$ is $0$-connected, then $A_\TQ^\wedge$ is as well, which implies \cite[2.7]{Ching_Harper_derived_Koszul_duality} that $A_\TQ^\wedge$ is $\TQ$-complete. A similar retract argument then shows that, in this case, the natural map $A \to A_\TQ^\wedge$ induces an isomorphism on $\TQ$-homology, i.e., that $A$ is ``$\TQ$-good.''

\vspace{.3em}
\noindent
{\bf Acknowledgements.} The author would like to thank John E. Harper for his advice and mentorship, Duncan Clark for many helpful conversations, and Yu Zhang, Jake Blomquist, and Crichton Ogle for useful discussions. We are also grateful to an anonymous referee who provided a careful reading and helpful suggestions for the improvement of an earlier draft of this paper. The author was supported in part by National Science Foundation grant DMS-1547357 and the Simons Foundation: Collaboration Grants for Mathematicians \#638247.

\section{Outline of the main argument}\label{sec_outline_of_main_argument}
We will now outline the proof of Theorem \ref{thm_main_theorem}. The basic idea, motivated by \cite{Arone_Kankaanrinta}, is that it it is easier to analyze the Taylor towers of iterates $(UQ)^k$ of the stabilization functor $UQ \simeq \Omega^\infty\Sigma^\infty$ than it is to analyze the Taylor tower of the identity directly. In fact, since the Taylor tower construction plays nicely with finite homotopy limits, it is equally advantageous to consider finite homotopy limits of functors of the form $(UQ)^k$. To make this more precise, consider the canonical resolution
\begin{align}\label{eqn_canonical_tq_resolution}
\xymatrix{
\id \ar[r] & UQ \ar@<0.5ex>[r] \ar@<-0.5ex>[r] & (UQ)^2 \ar@<1ex>[r] \ar[r] \ar@<-1ex>[r] & (UQ)^3 \cdots
}
\end{align}
of the identity functor by $\TQ$-homology (see, for instance \cite[3.10]{Ching_Harper_derived_Koszul_duality}) and let $\TQ_n:=\holim_{\Delta^{\leq n}}(UQ)^{\bullet+1}$ where $\Delta^{\leq n}$ denotes the full subcategory of $\Delta$ with no objects above degree $n$. This gives rise to a map of towers of functors
\begin{align}\label{eqn_id_to_tq_n}
\sett{\id} \to \sett{\TQ_n}
\end{align}
in $\AlgO$, where the tower on the left is taken to be constant. The strategy now is to analyze the Taylor tower of the identity by analyzing the Taylor towers of each $\TQ_n$. That is, let $A$ be as in Theorem \ref{thm_main_theorem}, and consider the diagram of the form

\vspace{-1.2em}
\begin{align}\label{eqn_the_big_diagram}
\xymatrix@R=2em{
\vdots\ar[d]&
\vdots\ar[d]&
\vdots\ar[d]&
\null\\
P_2(\TQ_0)A\ar[d]&
P_2(\TQ_1)A\ar[d]\ar[l]&
P_2(\TQ_2)A\ar[d]\ar[l]&
\cdots\ar[l]\\
P_1(\TQ_0)A&
P_1(\TQ_1)A\ar[l]&
P_1(\TQ_2)A\ar[l]&
\cdots\ar[l]\\
\TQ_0A\ar@{.>} ^-{(\#)}[u]&
\TQ_1A\ar[l]\ar@{.>} ^-{(\#)}[u]&
\TQ_2A\ar[l]\ar@{.>} ^-{(\#)}[u]&
\cdots\ar[l]
}
\end{align}

\noindent
in $\AlgO$ obtained by forming Taylor towers above each functor $\TQ_n$ and let $P_\ast(\TQ_\ast)A$ denote the bi-tower in \eqref{eqn_the_big_diagram} formed by the solid arrows above the bottom row. By taking the homotopy limit of $P_\ast(\TQ_\ast)A$ in two different ways, we will obtain the desired weak equivalence $P_\infty(\id)A \simeq A_\TQ^\wedge$.

On one hand, it follows from Proposition \ref{prop_convergence_of_columns} that each of the maps $(\#)$ induce a weak equivalence after applying $\holim$. Hence, taking homotopy limits of each vertical tower in $P_\ast(\TQ_\ast)A$ yields a diagram of the form
\begin{align}
\xymatrix{
\TQ_0(A) & \ar[l] \TQ_1(A) & \TQ_2(A) \ar[l] & \cdots \ar[l]
}
\end{align}
the homotopy limit of which is $A^\wedge_\TQ$. Since taking homotopy limits of $P_\ast(\TQ_\ast)A$ ``vertically'' and then ``horizontally'' calculates $\holim P_\ast(\TQ_\ast)A$, this establishes the following weak equivalence.

\begin{align}\label{eqn_holim_big_is_completion}
\holim P_\ast(\TQ_\ast)A \simeq A_\TQ^\wedge
\end{align}

On the other hand, one can also calculate $\holim P_\ast(\TQ_\ast)A$ by taking the homotopy limit of its diagonal. We analyze this diagonal in Section \ref{sec_analysis_diagonal} and, in particular, show in Proposition \ref{prop_upshot_diagonal} that we have the following weak equivalences.
\begin{align}\label{eqn_holim_big_is_p_infty}
\holim P_\ast(\TQ_\ast)A \simeq \holim\diag P_\ast(\TQ_\ast)A \simeq P_\infty(\id)A
\end{align}
Together, \eqref{eqn_holim_big_is_completion} and \eqref{eqn_holim_big_is_p_infty} prove Theorem \ref{thm_main_theorem}.

\section{Background and preliminaries}\label{sec_background_preliminaries}

In this section, we give a brief review of $\TQ$-homology and completion, and establish the working assumptions used in the rest of the paper.

\begin{defn}
Let $\tau_1\capO$ be the operad defined levelwise by
	\begin{align*}(\tau_1\capO)[r] := 
	\begin{cases}
		\capO[r],& \text{for } r \leq 1,\\
		\hfill \ast,& \text{otherwise}
	\end{cases} 
	\end{align*}
\end{defn}

\noindent
The canonical truncation map $\capO \to \tau_1\capO$ induces a change of operads adjunction
\begin{align}\label{eq:barq_baru}
\xymatrix{
\AlgO \ar@<0.5ex>^-{\bar{Q}}[r] & \Alg_{\tau_1\capO} \ \iso \ \Mod_{\capO[1]} \ar@<0.5ex>^-{\bar{U}}[l]
}
\end{align}
with left adjoint on top, where $\bar{Q}(A):= \tau_1\capO\circ_\capO(A)$ and $\bar{U}$ is the forgetful functor. This is, in fact, a Quillen adjunction (see \cite{Harper_symmetric_spectra} and \cite{Harper_Hess}), and the $\TQ$-homology of an $\capO$-algebra is defined using this adjunction as follows.

\begin{defn}
The $\TQ$-homology of an $\capO$-algebra $A$ is $\TQ(A):= \mathsf{R}\bar{U}\left(\mathsf{L}\bar{Q}(A)\right)$, where $\mathsf{R}$ and $\mathsf{L}$ denote the left and right derived functors, respectively.
\end{defn}

The idea of the $\TQ$-completion of an $\capO$-algebra $A$ is to iterate the $\TQ$-Hurewicz map $A \to \bar{U}\bar{Q}A$ to build a cosimplicial ``$\TQ$-resolution'' of $A$. However, even if $A$ is cofibrant, its $\TQ$-homology $\bar{U}\bar{Q}A$ need not be, and so such a construction would not necessarily be homotopical. In order to build an iterable model of $\TQ$-homology, an additional step is required. Following \cite[3.16]{Harper_Hess}, we first factor the map of operads $\capO \to \tau_1\capO$ as
\begin{align*}
\capO \to J \to \tau_1\capO
\end{align*}
a cofibration followed by a weak equivalence. This induces the following Quillen adjunctions
\begin{align}\label{eq:O_J_tau_adjunctions}
\xymatrix{
	\AlgO \ar^-{Q}@<.5ex>[r] & \AlgJ \ar@<0.5ex>[r] \ar^-{U}@<0.5ex>[l] & \Alg_{\tau_1\capO} \ar@<0.5ex>[l]
}
\end{align}
where $Q(A) := J\circ_\capO(A)$ and $U$ is the forgetful functor. The right-hand adjunction is a Quillen equivalence (see \cite[7.21]{Harper_Hess}) and it follows that, for cofibrant $A$, one has a weak equivalance $\TQ(A) \simeq UQA$. The advantage of this construction is that under appropriate conditions (see Assumption \ref{assumption}) the functor $U \colon \AlgJ \to \AlgO$ preserves cofibrant objects. Hence, $UQA$ gives an iterable point-set model for the $\TQ$-homology of $A$, which can be used to construct its $\TQ$-completion.

\begin{defn}
Let $A$ be an $\capO$-algebra and $LA$ denote a functorial cofibrant replacement of $A$. The $\TQ$-completion of $A$ is defined as 
\begin{align}
\xymatrix{
A_\TQ^\wedge:=\holim_\Delta \big(
 (UQ)LA \ar@<0.5ex>[r] \ar@<-0.5ex>[r] & (UQ)^2LA \ar@<1ex>[r] \ar[r] \ar@<-1ex>[r] & (UQ)^3LA \cdots	
}
\big)
\end{align}
\end{defn}

Lastly, we introduce the following standing assumption. The connectivity conditions on $\capO$ and $\capR$ guarantee that the results of \cite{Ching_Harper} used throughout this paper apply, while the cofibrancy assumption on $\capO$ implies \cite[5.49]{Harper_Hess} that the forgetful functor $\AlgJ \to \AlgO$ sends cofibrant objects to cofibrant objects.

\begin{assumption}\label{assumption}
In this paper, $\capO$ will denote a reduced operad in the category $(\ModR, \wedge, \capR)$ of $\capR$-modules (see, e.g., \cite{Hovey_Shipley_Smith}, \cite{Schwede_book_project}, or \cite{Shipley_commutative_ring_spectra}). We assume that $\capR$ and each $\capO[n]$ is $(-1)$-connected and, furthermore, that $\capO$ satisfies the following cofibrancy condition, which appears also in \cite[2.1]{Ching_Harper_derived_Koszul_duality}. Consider the unit map $I \to \capO$; we assume, for each $r \geq 0$, that $I[r] \to \capO[r]$ is a flat stable cofibration (see \cite[7.7]{Harper_Hess}) between flat stable cofibrant objects in $\ModR$. Unless stated otherwise, we consider $\AlgO$ with the positive flat stable model structure \cite[7.14]{Harper_Hess} and assume all homotopy groups are derived \cite{Schwede_homotopy_groups, Schwede_book_project}.
\end{assumption}

\section{Functor calculus in $\AlgO$}

Because the proofs given in Sections \ref{sec_analysis_diagonal}, \ref{sec_analysis_columns}, and \ref{sec_big_proof} work so directly with the construction of the Taylor tower, we review, in this section, the necessary background of this construction. It is essentially a recapitulation of \cite{Pereira_general_context}, with minor changes, and the expert can safely skip this section. Throughout the remainder of this paper, $\capM$ and $\capN$ will denote arbitrary reduced operads in $\ModR$, and we choose a fixed functorial factorization on each of $\AlgM$ and $\AlgN$.

Many of the functors we consider in later sections are not on-the-nose homotopy functors, but do become so after precomposition with a functorial cofibrant replacement. Hence, we have the following.

\begin{defn}\label{defn_left_homotopical}
Call a functor $F \colon \AlgM \to \AlgN$ \emph{left homotopical} if it preserves weak equivalences between cofibrant objects.
\end{defn}

\begin{defn}
Given a set $W$, let $\mathcal{P}(W)$ denote the power set of $W$ and $\mathcal{P}_0(W)$ the set of all nonempty subsets of $W$. Note that $\mathcal{P}(W)$ and $\mathcal{P}_0(W)$ are naturally partially ordered by inclusion and, hence, are categories.
\end{defn}
\begin{rem}
If $W$ is finite, then the simplicial nerve of $\mathcal{P}(W)$ has finitely many nondegenerate simplices, i.e., $\mathcal{P}(W)$ is ``very small'' in the language of \cite{Dwyer_Spalinski}.
\end{rem}
\begin{defn}\label{defn_goodwillie_n_cube}
Let $W=\sett{1,2,\ldots, n}$. Given an algebra $A$ in $\AlgM$, let $\mathcal{X}^n(A)$ denote the $\mathcal{P}(W)$-shaped diagram in $\Alg$ obtained via left Kan extension of the diagram
\begin{align}\label{eqn_def_cube_by_kan_extension}
\xymatrix@R=1em@C=1em{
&&A \ar[dll] \ar[dl] \ar[dr] \ar[drr]&&\\
C(A) & C(A) & \cdots & C(A) & C(A)
}
\end{align}
indexed on subsets of $W$ of cardinality less than or equal to one. Here, $C(A)$ denotes the cone on $A$ obtained by functorial factorization of the map from $A$ to the terminal object as a cofibration followed by an acyclic fibration. 
\end{defn}

\begin{rem}
One could equivalently define $\mathcal{X}^n(A)$ as the $n$-cube in $\AlgM$ obtained by taking iterated pushouts of the maps in \eqref{eqn_def_cube_by_kan_extension}. In particular, $\mathcal{X}^n(A)$ is a pushout cube in the language of \cite{Goodwillie_calculus_2}. Furthermore, if $A$ is cofibrant, then each of the maps in \eqref{eqn_def_cube_by_kan_extension} is a cofibration and, hence, the $n$-cube $\mathcal{X}^n(A)$ is strongly cocartesian. 
\end{rem}

\begin{rem}
As defined above, $C(A)$ depends on our fixed choice of functorial factorization. If needed, one could instead define the cone on $A$ using the pointed simplicial model structure \cite[6.1]{Harper_Hess} on $\AlgM$, which would result in a weakly equivalent model of $C(A)$ and would not depend on any such choice. This approach is taken in \cite{Pereira_general_context} to obtain certain universality results; in particular, see \cite[4.15]{Pereira_general_context}. In the interest of keeping this paper more self-contained, we avoid appealing to this extra simplicial structure.
\end{rem}

We now define the $n^{th}$ Taylor approximation of a functor $F \colon \AlgM \to \AlgN$ as in \cite{Pereira_general_context}, except that we define $\bar{T}_n$ and $\bar{P}_n$ for all such functors, rather than only homotopy functors that take values in cofibrant objects. The reason for this is that we would like to apply these constructions to, e.g., left homotopical functors, which can be made into homotopy functors with values in cofibrant objects via appropriate replacements.

\begin{defn}\label{defn_bar_T_n_F}
Given a functor $F \colon \AlgM\ \to \AlgN$, define $\bar{T}_nF$ objectwise by functorial factorization
\begin{align}\label{eqn_def_of_T_n}
F(A) \xrightarrow{(\bar{t}_nF)(A)} \bar{T}_nF(A) \xrightarrow{\sim} \holim_{\mathcal{P}_0(W)} F\big( \mathcal {X}^{n+1}(A)\big)
\end{align}
as a cofibration followed by an acyclic fibration.
\end{defn}

The $\bar{T}_nF$ construction of Definition \ref{defn_bar_T_n_F} takes the place of Goodwillie's $T_nF$ in \cite{Goodwillie_calculus_3}. Accordingly, $\bar{P}_nF$ is defined by the analogous sequential homotopy colimit.

\begin{defn}\label{defn_bar_P_n_F}
Given a functor $F \colon \AlgM \to \AlgN$, define $\bar{P}_n(F)A$ as the homotopy colimit of the following diagram.
\begin{align}
\bar{T}_n(F)A \xrightarrow{(\bar{t}_nF)(A)} \bar{T}_n\big(\bar{T}_n(F)\big)A \xrightarrow{(\bar{t}_n\bar{T}_nF)(A)} (\bar{T}_n^3F)A \to \cdots
\end{align}
\end{defn}

\begin{rem}
It is worth noting that, even if $F$ is a homotopy functor, the constructions $\bar{T}_nF$ and $\bar{P}_nF$ are \emph{not} homotopical. This is because, in general, the construction $\mathcal{X}^n$ of Definition \ref{defn_goodwillie_n_cube} is only homotopical on cofibrant inputs. However, to make these constructions homotopical, we only need to add in appropriate cofibrant replacement.
\end{rem}

\begin{defn}
Define $P_n(F)$ as $P_n(F):= \bar{P}_n(F)\circ L$, where $L$ denotes our fixed functorial cofibrant replacement in $\AlgM$.
\end{defn}

The following proposition says that if $F$ is a left homotopical functor, then, up to weak equivalence, there is no difference between working with $P_n(F)$ and $\bar{P}_n(F)$.

\begin{lem}\label{lem_T_n_F_left_homotopical}
If $F$ is left homotopical, then the natural map $P_n(F)A \xrightarrow{\sim} \bar{P}_n(F)A$ is a weak equivalence on all cofibrant objects $A$.
\end{lem}
\begin{proof}
It follows by construction that if $F$ is left homotopical, the functors $\bar{T}_n^k(F)$ are also left homotopical, for $k \geq 1$. Hence, $\hocolim_i\bar{T}_n^k(F) = \bar{P}_n(F)$ is as well. So, if $LA \xrightarrow{\sim} A$ is the functorial cofibrant replacement of $A$, we have that
\begin{align}
P_n(F)A = \bar{P}_n(F)LA \xrightarrow{\sim} \bar{P}_n(F)A
\end{align}
is a weak equivalence.
\end{proof}

The following two propositions follow from the construction of $\mathcal{X}^n(A)$ and the fact \cite[3.8(a)]{Ching_Harper} that homotopy pushouts of $k$-connected maps are $k$-connected. They mirror the fact that, in spaces, joining with a nonempty set increases connectivity of maps and preserves strongly cocartesian cubes (see, for instance, the proof of \cite[1.4]{Goodwillie_calculus_2}). Note that for a cube $\mathcal{X}$ indexed on $\sett{1,2,\ldots, n}$ and any subset $U \subseteq \sett{1,2,\ldots, n}$, we denote by $\mathcal{X}_U$ the object of $\mathcal{X}$ at index $U$.

\begin{prop}\label{prop_join_increase_connectivity}
Given a $k$-connected map $A \to B$ between cofibrant objects in $\AlgM$ and any nonempty subset $U \subseteq \sett{1,2,\ldots, n}$, the induced map $\mathcal{X}^n_U(A) \to \mathcal{X}^n_U(B)$ is $(k+1)$-connected.
\end{prop}

\begin{prop}\label{prop_join_preserves_strongly_cocartesian}
If $\mathcal{Y}$ is an objectwise cofibrant, strongly cocartesian cube in $\AlgM$, then for any fixed $U \subseteq \sett{1,2,\ldots,n}$, the cube $\mathcal{X}_U^n(\mathcal{Y})$ is strongly cocartesian.
\end{prop}

\section{Analysis of the diagonal of $P_\ast(\TQ_\ast)A$}\label{sec_analysis_diagonal}

The purpose of this section is to prove Proposition \ref{prop_upshot_diagonal}. After observing that the analogue of \cite[1.6]{Goodwillie_calculus_3} holds in our setting, the key technical maneuver is Proposition \ref{prop_id_tq_n_1_agree_order_n}, which is a consequence of \cite[7.1]{Blomquist_iterated_delooping}.

The following definition is essentially \cite[1.2]{Goodwillie_calculus_3}, but restricted to cofibrant objects. Since we make frequent use of this definition, we have included it for the sake of completeness.

\begin{defn}
A map $F \to G$ between functors $\AlgM \to \AlgN$ is said to satisfy $O_n'(c, \kappa)$ if for every cofibrant $A \in \AlgM$ such that $A \to \ast$ is $k$-connected, with $k \geq \kappa$, the map $F(A) \to G(A)$ is $(-c+(n+1)k)$-connected. We say that $F$ and $G$ \emph{agree to order $n$ on cofibrant objects} if this holds for some constants $c$ and $\kappa$.
\end{defn}

\begin{prop}\label{prop_if_functors_agree}
If a map $F \to G$ between functors $\AlgM \to \AlgN$ satisfies $O_n'(c, \kappa)$, then
\begin{enumerate}[label=(\roman*), itemsep=2pt]
\item $\bar{T}_n^i(F) \to \bar{T}_n^i(G)$ satisfies $O_n'(c-i, \kappa -i)$
\item $\bar{P}_n(F)A \xrightarrow{\sim} \bar{P}_n(G)A$ is a weak equivalence for $(-1)$-connected, cofibrant $A$ 
\end{enumerate}
\end{prop}
\begin{proof}
Using Propositions \ref{prop_join_increase_connectivity} and \ref{prop_join_preserves_strongly_cocartesian}, this is proven as in \cite[1.6]{Goodwillie_calculus_3}.	
\end{proof}

We are now in a position to prove the key technical result of this section, which states that the identity functor agrees to order $n$ on cofibrant objects with the functor $\TQ_{n-1}$. To keep this section appropriately brief, we refer the reader to \cite[Section 7]{Schonsheck_fibration_theorems} for the relevant background on cubical diagrams used below.

\begin{prop}\label{prop_id_tq_n_1_agree_order_n}
The natural transformation $\id \to \TQ_{n-1}$ of functors $\AlgO \to \AlgO$ satisfies $O_n'(0,1)$ for all $n \geq 1$.
\end{prop}
\begin{proof}
The proposition essentially follows from \cite[7.1]{Blomquist_iterated_delooping}. For instance, if $A \in \AlgO$ is cofibrant and $k$-connected with $k \geq 0$, consider the coface cubes 
\begin{align}\label{eqn_coface_cubes}
\xymatrix@C=.6em@R=.7em{
A \ar[rr] && UQA && A \ar[rr] \ar[dd] && UQA \ar[dd]&&A \ar[dd] \ar[dr] \ar[rr] && UQA \ar[dd]|\hole \ar[dr]\\
&&&&&&&&& UQA \ar[dd] \ar[rr] && (UQ)^2A \ar[dd]\\
&&&&UQA \ar[rr] && (UQ)^2A&&UQA\ar[rr]|\hole \ar[dr]&& (UQ)^2A \ar[dr]\\
&&&&&&&&& (UQ)^2A \ar[rr] && (UQ)^3A
}
\end{align}
arising from the cosimplicial $\TQ$-resolution of $A$. As a $0$-cube, $A$ is $(k+1)$-cartesian; repeated application of \cite[7.1]{Blomquist_iterated_delooping} then shows that the cubes in \eqref{eqn_coface_cubes} are, respectively, $2(k+1)$-cartesian, $3(k+1)$-cartesian, and $4(k+1)$-cartesian. Hence, the map $A \to \TQ_0(A)$ is $2(k+1)$-connected; the map $A \to \TQ_1(A)$ is $3(k+1)$-connected; and the map $A \to \TQ_2(A)$ is $4(k+1)$-connected. To finish the proof, continue inductively.
\end{proof}

The main result of this section, below, is now a straightforward consequence of Proposition \ref{prop_id_tq_n_1_agree_order_n}.

\begin{prop}\label{prop_upshot_diagonal}
Given a $(-1)$-connected cofibrant $\capO$-algebra $A$, there is a natural weak equivalence of the form $\holim P_\ast(\TQ_\ast)A \simeq P_\infty(\id)A$.
\end{prop}
\begin{proof}
It follows from Propositions \ref{prop_if_functors_agree} and Proposition \ref{prop_id_tq_n_1_agree_order_n} that there is a natural weak equivalence between the diagonal of $P_\ast(\TQ_\ast)A$ and the tower $\sett{P_n(\id)}$. Since the inclusion of the diagonal into a bi-tower is left cofinal (i.e., homotopy initial), it follows that we have weak equivalences
\begin{align}
\holim P_\ast(\TQ_\ast)A \simeq \holim\diag P_\ast(\TQ_\ast)A \simeq \holim_nP_n(\id)A =:P_\infty(\id)A
\end{align}
and this completes the proof.
\end{proof}

\section{Analysis of the columns of $P_\ast(\TQ_\ast)A$}\label{sec_analysis_columns}

The purpose of this section is to prove Proposition \ref{prop_convergence_of_columns}. The key techincal result used is Proposition \ref{prop_technical_heart}, the proof of which we postpone to Section \ref{sec_big_proof} in order to clarify the overall argument. In the results below, we make use the fact that, by construction, the functors $Q$, $(UQ)^mU$, and $\TQ_m$ are left homotopical for all $m \geq 0$.

As in the previous section, we give the following definition and proposition for the sake of completeness; they are the analogues of \cite[4.1]{Goodwillie_calculus_2} and \cite[1.4]{Goodwillie_calculus_3}, respectively.

\begin{defn}
Given $F\colon\AlgM \to \AlgN$, we say $F$ is \emph{cofibrantly stably $n$-excisive} or satisfies \emph{cofibrant stable $n^{th}$ order excision} if the following is true for some numbers $c$ and $\kappa$.

\begin{addmargin}[2em]{2em}
$E_n'(c, \kappa)$: If $\mathcal{X} \colon P(S) \to \AlgM$ is an objectwise cofibrant, objectwise $(-1)$-connected, strongly cocartesian $(n+1)$-cube such that for all $s \in S$ the map $\mathcal{X}_\emptyset \to \mathcal{X}_s$ is $k_s$-connected and $k_s \geq \kappa$, then the diagram $F(\mathcal{X})$ is $(-c+\sum k_s)$-cartesian.
\end{addmargin}
\end{defn}

\begin{prop}\label{prop_if_f_satisfies_stable_excision}
If $F \colon \AlgM \to \AlgN$ satisfies $E_n'(c, \kappa)$, then
\begin{enumerate}[label=(\roman*)]
\item $\bar{T}_nF$ satisfies $E_n'(c-1, \kappa-1)$
\item $\bar{t}_nF\colon F \to \bar{T}_nF$ satisfies $O_n'(c, \kappa)$
\end{enumerate}
\end{prop}
\begin{proof}
Using Propositions \ref{prop_join_increase_connectivity} and \ref{prop_join_preserves_strongly_cocartesian}, this is proven in the same way as \cite[1.4]{Goodwillie_calculus_3} (replacing \cite[1.20]{Goodwillie_calculus_2} equivalently with repeated application of \cite[3.8]{Ching_Harper}).
\end{proof}

The following proposition is the key ingredient used to prove Proposition \ref{prop_convergence_of_columns}. Since the proof of the following result is somewhat technical, we have deferred it to Section \ref{sec_big_proof}.

\begin{prop}\label{prop_technical_heart}
For each $m \geq 0$, the functor $(UQ)^mU \colon \AlgJ \to \AlgO$ satisfies $E_n'(-1, 0)$ for all $n \geq 1$.
\end{prop}

We analyze the functors $(UQ)^mU$ rather than $(UQ)^{m+1}$ because the former make it easier to use the assumption of our main result that $\TQ(A)$ is $0$-connected. The following lemma says that, up to weak equivalence, there is no difference in analyzing the Taylor towers of the two functors.

\begin{lem}\label{lem_t_nfq_vs_t_nf_q}
For any left homotopical functor $F$ and cofibrant $A \in \AlgO$, there is a natural commutative diagram
\begin{align}\label{eqn_t_nfq_vs_t_nf_q}
\xymatrix{
(FQ)A \ar[r] \ar[dr] & \bar{T}_n^i(FQ)A \ar ^-{\sim}[d]\\
& \bar{T}_n^i(F)QA
}
\end{align}
for all $n, i \geq 1$.
\end{lem}
\begin{proof}
Since $\mathcal{X}^{n+1}(A)$ is a pushout cube and $Q$ is a left adjoint, there is a natural isomorphism $Q\mathcal{X}^{n+1}(A) \ \iso \ \mathcal{X}^{n+1}(QA)$. This proves the lemma for $i =1$, and one then continues inductively.
\end{proof}

\begin{rem}
Applying $\hocolim_i$ to \eqref{eqn_t_nfq_vs_t_nf_q} shows that Lemma \ref{lem_t_nfq_vs_t_nf_q} remains true if $\bar{T}_n^i$ is replaced by $\bar{P}_n$.
\end{rem}

\noindent
We will now work towards proving Proposition \ref{prop_convergence_of_columns}, assuming Proposition \ref{prop_technical_heart}.

\begin{prop}\label{prop_uqm_to_tnuqm}
For all $n, i \geq 1$ and $m \geq 0$, if $A \in \AlgO$ is cofibrant and $QA$ is $0$-connected, then the map 
\begin{align}
(UQ)^{m+1}A \to \bar{T}_n^i\big((UQ)^{m+1}	\big)A
\end{align}
is $(n+2)$-connected.
\end{prop}
\begin{proof}
It follows from Propositions \ref{prop_if_f_satisfies_stable_excision} and \ref{prop_technical_heart} that the map $(UQ)^mUQA \to \bar{T}_n\big( (UQ)^mU \big)QA$ is $(n+2)$-connected. Inductive application of Proposition \ref{prop_if_f_satisfies_stable_excision} then shows that, for $i \geq 1$, the maps $\bar{T}_n^i\big((UQ)^mU\big)QA \to \bar{T}_n^{i+1}\big(	 (UQ)^mU	\big)QA$ are also at least $(n+2)$-connected. Since $\pi_\ast$ commutes with filtered homotopy colimits, this means that, for all $i \geq 1$, the map $(UQ)^mUQA \to \bar{T}_n^i\big(	 (UQ)^mU	\big)QA$ is $(n+2)$-connected as well. The result then follows from Lemma \ref{lem_t_nfq_vs_t_nf_q} with $F = (UQ)^mU$.
\end{proof}

The connectivity estimates provided by Proposition \ref{prop_uqm_to_tnuqm} now translate to similar estimates for the functors $\TQ_n$ which we use to prove the following.

\begin{prop}\label{prop_convergence_of_columns}
For any $k \geq 0$, the natural map 
\begin{align}
\TQ_kA \to \bar{P}_\infty\big(	\TQ_k\big)A
\end{align}
is a weak equivalence.
\end{prop}
\begin{proof}
If $k=0$, this follows from the fact that $\TQ_0 = UQ$ is itself an excisive functor.

We now give the proof for $k=1$ as it is instructive in understanding the more general argument. First, we use the fact that $\TQ_1A$ can be calculated as the homotopy limit of a punctured $2$-cube. In more detail, there is a canonical map
\begin{equation}
\big( [0] \rightarrow [1] \leftarrow [0] \big) \quad \to  \quad \Delta^{\leq 1}
\end{equation}
which is left cofinal, i.e., homotopy initial (see \cite[Section 6]{Carlsson}, \cite{Dugger_homotopy_colimits}, or \cite[6.7]{Sinha_cosimplicial_models}). Applied to our case, this means that the back face of the following diagram

\begin{align}\label{eqn_tq_k_to_T_n_tq_k}
\xymatrix{
\TQ_1A \ar[dd] \ar[dr] \ar[rr] && UQA \ar[dd]|\hole \ar^-{n+2}[dr]\\
& \bar{T}_{n}^i(\TQ_1)A \ar[dd] \ar[rr] && \bar{T}_{n}^i(UQ)A \ar[dd]\\
UQA \ar[rr]|\hole \ar^-{n+2}[dr]&& (UQ)^2A \ar^-{n+2}[dr]\\
& \bar{T}_{n}^i(UQ)A\ar[rr] && \bar{T}_{n}^i(UQ)^2A
}
\end{align}
is a homotopy pullback square. For any $n \geq 1$, the indicated maps are $(n+2)$-connected by Proposition \ref{prop_uqm_to_tnuqm} and, since $\bar{T}_{n}^i$ commutes with homotopy limits, the front face of \eqref{eqn_tq_k_to_T_n_tq_k} is also a homotopy pullback square. As both the back and front faces of \eqref{eqn_tq_k_to_T_n_tq_k} are cartesian, it follows that the entire $3$-cube is cartesian. Repeated application of \cite[3.8]{Ching_Harper} then implies that the map $\TQ_1A \to \bar{T}_{n}^i(\TQ_1)A$ is $(n+1)$-connected. Since $\pi_\ast$ commutes with filtered homotopy colimits, this implies that the map 
\begin{align}
\TQ_1A \to \bar{P}_{n}\Big(\TQ_1	\Big)A
\end{align}
is also $(n+1)$-connected. The result now follows by letting $n$ tend to infinity and considering the associated $\lim^1$ short exact sequence.

The argument for $k \geq 1$ is similar. We first appeal to the fact that $(\TQ_k)A$ can be calculated as the homotopy limit of a punctured $(k+1)$-cube and consider the map $\TQ_kA \to \bar{T}_{n+k-1}^i(\TQ_k)A$. We construct a $(k+2)$-cube analogous to \eqref{eqn_tq_k_to_T_n_tq_k} from the maps $(UQ)^mA \to \bar{T}_{n+k-1}^i(UQ)^mA$, which are $(n+k+1)$-connected by Proposition \ref{prop_uqm_to_tnuqm}. Repeated application of \cite[3.8]{Ching_Harper} (or appealing to the proof of \cite[1.20]{Goodwillie_calculus_1}) shows that the map $\TQ_kA \to \bar{T}_{n+k-1}^i(\TQ_k)A$ is $(n+1)$-connected. As in the $k=1$ case, it then follows that the map

\begin{equation}
\TQ_kA \to \bar{P}_{n+k-1}(\TQ_k)A
\end{equation}
is $(n+1)$-connected. Letting $n$ tend to infinity and analyzing the associated $\lim^1$ short exact sequence completes the proof. 
\end{proof}

\section{A sufficient condition}\label{sec_sufficient_condition}

The purpose of this section is to prove Theorem \ref{thm_sufficient_condition}. The strategy is to, first, use the fact \cite[4.10]{Harper_Hess} that $\TQ(A)$ can be calculated as the realization of a simplicial bar construction, i.e., $\TQ(A)\simeq\abs{\BAR(\tau_1\capO, \capO,A)}$. Theorem \ref{thm_sufficient_condition} then follows by analyzing the spectral sequence \cite[4.43]{Harper_Hess} arising from the filtration by simplicial skeleta of the associated bar construction. For further details on the construction of this spectral sequence, see, for instance, \cite[X.2.9]{EKMM}.

\begin{proof}[Proof of Theorem \ref{thm_sufficient_condition}]
Let $A$ be a $(-1)$-connected cofibrant $\capO$-algebra and suppose the following.
\begin{align*}
\text{($\ast$) : For some $k \geq 2$, the map $\capO[k]\wedge_{\Sigma_k}A^{\wedge k} \to A$ induces a surjection on $\pi_0$.}
\end{align*} 

Consider the simplicial $\capR$-module $\BAR(\tau_1\capO,\capO,A)$ and associated spectral sequence, as below.
\begin{align}\label{eqn_bar_spectral_sequence}
E^2_{p,q}=H_p\big(\pi_q\big(\BAR(\tau_1\capO,\capO,A)\big)	\big) \implies \pi_{p+q}(\abs{\BAR(\tau_1\capO,\capO,A)}) 
\end{align}
Along with our standing connectivity condition on $\capO$ (see Assumption \ref{assumption}) the fact that $A$ is $(-1)$-connected implies that $\BAR(\tau_1\capO,\capO,A)$ is objectwise $(-1)$-connected. Hence, the only contribution to $\pi_0(\abs{\BAR(\tau_1\capO,\capO,A)})$ comes from $p=q=0$. Furthermore, for dimensional reasons, we have $E^2_{0,0} = E^\infty_{0,0}$. The result will therefore follow by showing that $E^2_{0,0} = \ast$.

It follows from \eqref{eqn_bar_spectral_sequence} and properties of simplicial abelian groups (see, e.g., \cite[III.2]{Goerss_Jardine}) that
\begin{equation}\label{eqn_calculation_1}
H_0\big(\pi_0\big(\BAR(\tau_1\capO,\capO,A)\big)	\big) \ \iso \ \pi_0(\tau_1\capO\circ(A))/\im(d_0-d_1)\\
\end{equation}
We will now show that $(\ast)$ implies that the map
\begin{equation}
\begin{gathered}
\pi_0\big(\tau_1\capO\circ\capO\circ(A)) \xrightarrow{d_0-d_1} \pi_0(\tau_1\capO\circ(A)) \ \ \iso \\
\pi_0\big( \capO[1] \wedge \coprod_{k=1}^\infty\capO[k]\wedge_{\Sigma_k}A^{\wedge k}\big) \xrightarrow{d_0-d_1} \pi_0\big( \capO[1] \wedge A)
\end{gathered}
\end{equation}
is surjective. Indeed, the cofibrancy condition of Assumption \ref{assumption} guarantees that $\capO[1]$ is flat in the language of \cite{Schwede_book_project} and, hence, it suffices by \cite[5.23]{Schwede_book_project} to show that 
\begin{align}
\pi_0\Big( 	\coprod_{k=1}^\infty\capO[k]\wedge_{\Sigma_k}A^{\wedge k}\Big) \ \iso \ \bigoplus_{k=1}^\infty \pi_0\big( \capO[k]\wedge_{\Sigma_k}A^{\wedge k}\big) \xrightarrow{d_0-d_1} \pi_0(A)
\end{align}
is surjective. Since the map $d_0$ is trivial when $k \geq 2$, the result now follows immediately from $(\ast)$.
\end{proof}

\section{Proof of Proposition \ref{prop_technical_heart}}\label{sec_big_proof}
The purpose of this section is to prove Proposition \ref{prop_technical_heart}. The strategy is a double induction, on $n$ and $m$. More precisely, for fixed $m \geq 0$ and $n \geq 1$, let $P(n, m)$ be the following statement:
\begin{equation}\label{eqn_p_n_m}
\text{$P(n, m)$: The functor $(UQ)^mU \colon \AlgJ \to \AlgO$ satisfies $E_n'(-1,0)$.}
\end{equation}
Lemmas \ref{lem_p_1_m}, \ref{lem_p_2_m}, and \ref{lem_p_n_m} show that $P(n,m)$ holds for all $m \geq 0$ and $n \geq 1$, and this proves Proposition \ref{prop_technical_heart}. We have included an extra base case to better illustrate, in a low dimension, the more general inductive step. Throughout the arguments that follow, we make frequent use of the fact \cite[4.4]{Schonsheck_fibration_theorems} that the forgetful functor $U \colon \AlgJ \to \AlgO$ preserves cartesianness of all cubes, and that \cite[4.5]{Schonsheck_fibration_theorems} the functor $Q \colon \AlgO \to \AlgJ$ preserves cocartesianness of objectwise cofibrant and $(-1)$-connected cubes.

\begin{lem}\label{lem_p_1_m}
$P(1,m)$ holds for all $m \geq 0$.
\end{lem}
\begin{proof}
Let $\mathcal{Y}$ be an objectwise cofibrant, objectwise $(-1)$-connected, cocartesian $2$-cube
\begin{align}
\xymatrix{
Y_\emptyset \ar ^-{k_1}[r] \ar ^-{k_2}[d] & Y_{\sett{1}} \ar[d]\\
Y_{\sett{2}} \ar[r] & Y_{\sett{1,2}}
}
\end{align}
in $\AlgJ$, with initial maps of the indicated connectivities, and $k_1, k_2 \geq 0$. Since $\mathcal{Y}$ is a cocartesian square in the stable category $\AlgJ$, it is also cartesian. Hence, so is $U\mathcal{Y}$; this proves $P(1, 0)$.

To see that $P(1, 1)$ holds, note that $Y_{\sett{2}}\to Y_{\sett{1,2}}$ is $k_1$-connected and $Y_{\sett{1}} \to Y_{\sett{1,2}}$ is $k_2$-connected, by \cite[3.8]{Ching_Harper}. It follows from \cite[1.8]{Ching_Harper} that $U\mathcal{Y}$ is $(2+k_1+k_2)$-cocartesian and, hence, $QU\mathcal{Y}$ is as well. By \cite[3.10]{Ching_Harper}, this means that $QU\mathcal{Y}$ is $(1+k_1+k_2)$-cartesian and, hence, so is $UQU\mathcal{Y}$; this verifies $P(1,1)$.

Inductively, suppose that $P(1,m)$ holds for fixed $m \geq 1$. Let's show that $P(1,m+1)$ is true. By the inductive hypothesis, the cube $(UQ)^mU\mathcal{Y}$ is $(1+k_1+k_2)$-cartesian. Using now the dual Blakers-Massey theorem \cite[1.9]{Ching_Harper} for $\AlgO$, we have that this cube is $(2+k_1+k_2)$-cocartesian and, hence, so is $Q(UQ)^mU\mathcal{Y}$. By \cite[3.10]{Ching_Harper}, this means $Q(UQ)^mU\mathcal{Y}$ is $(1+k_1+k_2)$-cartesian, so $UQ(UQ)^mU\mathcal{Y}=(UQ)^{m+1}U\mathcal{Y}$ is as well.
\end{proof}

\begin{lem}\label{lem_p_2_m}
$P(2, m)$ holds for all $m \geq 0$.
\end{lem}
\begin{proof}
Let $\mathcal{Y}$ be an objectwise cofibrant, objectwise $(-1)$-connected, strongly cocartesian $3$-cube
\begin{align}
 \xymatrix @C=1em@R=1em{
\mathcal{Y}_\emptyset \ar ^-{k_1}[rr] \ar ^-{k_2}[dd] \ar ^-{k_3}[dr] && \mathcal{Y}_{\sett{1}}\ar@{.>}[dd] \ar [dr]\\
&\mathcal{Y}_{\sett{3}} \ar[rr] \ar[dd] && \mathcal{Y}_{\sett{1,3}} \ar[dd]\\
\mathcal{Y}_{\sett{2}} \ar[rr] \ar [dr] && \mathcal{Y}_{\sett{1,2}} \ar@{.>}[dr]\\
& \mathcal{Y}_{\sett{2,3}} \ar[rr] && \mathcal{Y}_{\sett{1,2,3}}
}
\end{align}
in $\AlgJ$, with $k_1, k_2, k_3 \geq 0$. Arguing as in Lemma \ref{lem_p_1_m} and replacing \cite[1.8]{Ching_Harper} by \cite[1.10]{Ching_Harper}, shows that $P(2,0)$ and $P(2,1)$ hold.

Inductively, suppose that $P(2, m)$ is true for fixed $m \geq 1$. Let's show that $P(2, m+1)$ is true. By the inductive hypothesis, the cube $(UQ)^mU\mathcal{Y}$ is $(1+k_1+k_2+k_3)$-cartesian. As in Lemma \ref{lem_p_1_m}, we first establish a cocartesian estimate on $(UQ)^mU\mathcal{Y}$, but now use the higher dual Blakers-Massey theorem \cite[1.11]{Ching_Harper} of Ching-Harper. Note that applying Lemma \ref{lem_p_1_m} to the subcubes of $(UQ)^mU\mathcal{Y}$ guarantees that the conditions of this theorem are satisfied. 

We will now adopt the notation of \cite[1.11]{Ching_Harper}. Since $P(1, k)$ holds for all $k \geq 0$, the minimum of $\sum_{V \in \lambda}k_V$ over all partitions $\lambda$ of $\sett{1,2,3}$ by nonempty sets not equal to $\sett{1,2,3}$ is achieved by the partition of singleton sets. It then follows from \cite[1.11]{Ching_Harper} that $(UQ)^mU\mathcal{Y}$ is $(3+k_1+k_2+k_3)$-cocartesian and, hence, so is $Q(UQ)^mU\mathcal{Y}$. By \cite[3.10]{Ching_Harper}, this means $Q(UQ)^mU\mathcal{Y}$ is $(1+k_1+k_2+k_3)$-cartesian, so $UQ(UQ)^mU\mathcal{Y} = (UQ)^{m+1}U\mathcal{Y}$ is as well.
\end{proof}

\begin{lem}\label{lem_p_n_m}
Fix $n \geq 2$. Assume, for all $1 \leq j \leq n$, that $P(j, m)$ holds for all $m \geq 0$. Then $P(n+1, m)$ also holds for all $m \geq 0$.
\end{lem}
\begin{proof}
Let $\mathcal{Y}$ be an objectwise cofibrant, objectwise $(-1)$-connected, strongly cocartesian $W$-cube in $\AlgJ$ with $W = \sett{1, 2, \ldots, n+2}$ such that for each $s \in W$, the map $Y_\emptyset \to Y_{\sett{s}}$ is $k_s$-connected, with $k_s \geq 0$. Arguing as in Lemma \ref{lem_p_2_m} shows that $P(n+1, 0)$ and $P(n+1,1)$ hold.

Inductively, suppose that $P(n+1, m)$ holds for fixed $m \geq 1$. To show that $P(n+1, m+1)$ holds as well, the strategy is the same as in Lemma \ref{lem_p_2_m}. That is, we first establish a cocartesian estimate on $(UQ)^mU\mathcal{Y}$, which we know by the inductive hypothesis is $(1+\sum_{s \in W}k_s)$-cartesian. 

The assumption that $P(j,m)$ holds for all $1 \leq j \leq n$ and $m \geq 0$ guarantees that the conditions of \cite[1.11]{Ching_Harper} are satisfied. Furthermore, as in Lemma \ref{lem_p_2_m}, this assumption implies that the minimum of $\sum_{V \in \lambda}k_V$ over all partitions $\lambda$ of $W$ by nonempty sets not equal to $W$ is achieved by the partition of singleton sets. It then follows from \cite[1.11]{Ching_Harper} that $(UQ)^mU\mathcal{Y}$ is $(\abs{W}+\sum_{s \in W}k_s)$-cocartesian and, hence, so is $Q(UQ)^mU\mathcal{Y}$. By \cite[3.10]{Ching_Harper}, this means $Q(UQ)^mU\mathcal{Y}$ is $(1+\sum_{s \in W}k_s)$-cartesian, so $UQ(UQ)^mU\mathcal{Y} = (UQ)^{m+1}U\mathcal{Y}$ is as well.
\end{proof}

\bibliographystyle{plain}
\bibliography{aConvergence_bibliography}

\end{document}